\documentclass[11pt]{amsart}

\usepackage{amsthm, amssymb,amsmath,latexsym, amscd,  mathtools}
\input cyracc.def

\newtheorem{theorem}{Theorem}[section]
\newtheorem{lemma}[theorem]{Lemma}
\newtheorem{proposition}[theorem]{Proposition}

\newtheorem{corollary}[theorem]{Corollary}
\theoremstyle{remark}
\newtheorem{remark}[theorem]{Remark}
\newtheorem{example}[theorem]{Example}
\newtheorem{question}[theorem]{Question}

\newcommand{\id}{\text{id}}
\newcommand{\lra}{\longrightarrow}

\newcommand{\sm}{\setminus}
\newcommand{\wt}[1]{\widetilde{#1}}

\newcommand{\Ob}{\operatorname{Ob}}

\def\ZZ{\mathbb{Z}}
\def\QQ{\mathbb{Q}}
\def\NN{\mathbb{N}}
\def\RR{\mathbb{R}}
\usepackage[margin=1in]{geometry}
\begin{document}

\title[Left-orderability from circular-orderability]{Promoting circular-orderability to left-orderability}

\date{\today}

\author[Jason Bell]{Jason Bell}
\thanks{Jason Bell was partially supported by NSERC grant RGPIN-2016-03632.}
\address{Department of Pure Mathematics\\
University of Waterloo \\
Waterloo \\
ON Canada N2L 3G1} \email{jpbell@uwaterloo.ca}
\urladdr{http://www.math.uwaterloo.ca/~jpbell/} 

\author[Adam Clay]{Adam Clay}
\thanks{Adam Clay was partially supported by NSERC grant RGPIN-2014-05465}
\address{Department of Mathematics\\
University of Manitoba \\
Winnipeg \\
MB Canada R3T 2N2} \email{Adam.Clay@umanitoba.ca}
\urladdr{http://server.math.umanitoba.ca/~claya/} 

\author[Tyrone Ghaswala]{Tyrone Ghaswala}
\thanks{Tyrone Ghaswala was partially supported by a PIMS Postdoctoral Fellowship at the University of Manitoba}
\address{Department of Mathematics\\
University of Manitoba \\
Winnipeg \\
MB Canada R3T 2N2} \email{ty.ghaswala@gmail.com}
\urladdr{https://server.math.umanitoba.ca/~ghaswalt/} 

\begin{abstract}
Motivated by recent activity in low-dimensional topology, we provide a new criterion for left-orderability of a group under the assumption that the group is circularly-orderable: A group $G$ is left-orderable if and only if $G \times \mathbb{Z}/n\mathbb{Z}$ is circularly-orderable for all $n>1$.  This implies that every circularly-orderable group which is not left-orderable gives rise to a collection of positive integers that exactly encode the obstruction to left-orderability, which we call the obstruction spectrum.  We precisely describe the behaviour of the obstruction spectrum with respect to torsion, and show that this same behaviour can be mirrored by torsion-free groups, whose obstruction spectra are in general more complex.
\end{abstract}

\maketitle

\section{Introduction}

A group $G$ is left-orderable if there exists a strict total ordering $<$ of its elements such that $g<h$ implies $fg<fh$ for all $f, g, h \in G$, and bi-orderable if $g<h$ implies $gf<hf$ also holds.   Equivalently, a group $G$ is left-orderable if there is a subset $P \subset G$, called the \emph{positive cone} of $G$, satisfying $P \cdot P \subset P$ and $P \sqcup P^{-1} = G \setminus \{id \}$.  For bi-orderability, one also requires $gPg^{-1} \subset P$ for all $g \in G$.

Various criteria exist that allow one to determine when $G$ is left-orderable or bi-orderable.  Notable examples are the Burns-Hale theorem \cite{BH72}, various semigroup conditions \cite{Conrad59, Los54, Ohnishi52}, or the existence of embeddings into left-orderable groups like $\mathrm{Homeo}_+(\mathbb{R})$ \cite{Conrad59}.  Still other criteria exist that describe when it is possible to use weaker structures on $G$ (such as locally invariant orderings or circular orderings \cite{LM12, CD03}) to assert the existence of a left ordering of $G$ by imposing certain algebraic assumptions.

This note presents a condition of the latter kind.  Specifically, we study the relationship between circular-orderability and left-orderability of a group, and determine necessary and sufficient conditions that a circularly-orderable group be left-orderable.  Results along these lines already exist in the literature (see Section \ref{background}), but all require the existence of a special type of circular ordering, from which one derives that the group at hand is in fact left-orderable.  In contrast, we arrive at left-orderability of a circularly-orderable group by imposing algebraic conditions on the group itself. We prove:

\begin{theorem}
\label{main theorem}
A group $G$ is left-orderable if and only if $G \times \ZZ/n \ZZ$ is circularly-orderable for all $n \geq 2$.
\end{theorem}

In light of this result, we define the \textit{obstruction spectrum} of a group to be the set of all $n \in \mathbb{N}_{>1}$ for which $G \times \ZZ/n \ZZ$ is not circularly-orderable.  As one might expect, the obstruction spectrum detects torsion---in the sense that every group containing torsion has nonempty obstruction spectrum, and for any group with torsion certain elements in its obstruction spectrum can be attributed to its torsion.  However the obstruction spectrum of a group does not detect torsion alone, as there are plenty of torsion-free circularly-orderable groups which are not left-orderable and thus have nonempty obstruction spectra. 

We therefore provide tools for computing the obstruction spectra of certain torsion-free examples: finitely generated amenable groups (Proposition \ref{daves reduction}), and free products (Proposition \ref{obstruction_free_products}).   Our computations show that: (1) the elements of any obstruction spectrum arising from torsion in the group can also be realized as the obstruction spectrum of a torsion-free group, and (2) the result of Theorem \ref{main theorem} is sharp, in the sense that for every $N \in \mathbb{N}$ there is a group $G$ such that the groups $\{G \times \ZZ/n \ZZ \mid n < N \}$ are all circularly-orderable, yet the obstruction spectrum of $G$ is nonempty.

Our motivation for studying circular-orderability and its relationship with left-orderability is two fold. First, for countable groups, the conditions of being left-orderable and circularly-orderable are equivalent to admitting injections into $\operatorname{Homeo}_+(\RR)$ and $\operatorname{Homeo}_+(S^1)$ respectively (see \cite[Theorem 2.2.14]{Calegari} for a proof).  Therefore, these combinatorial conditions completely characterize when countable groups act faithfully by orientation-preserving homeomorphisms on 1-manifolds. It is natural to ask what conditions we can put on a group that acts faithfully on $S^1$ to guarantee that it acts faithfully on $\RR$, or equivalently, to guarantee the existence of a faithful action on $S^1$ with a global fixed point. 

The other motivation comes from low-dimensional topology, where there has been substantial activity surrounding the so-called ``L-space conjecture".  For a given irreducible rational homology $3$-sphere $M$, this conjecture relates the properties of $\pi_1(M)$ being left-orderable to analytic (Heegaard-Floer theoretic) and topological (the existence of certain nice foliations) properties of $M$ \cite{BGW13, Ju15}.

At present, almost all approaches to left ordering $\pi_1(M)$ involve first finding a faithful action of $\pi_1(M)$ on $S^1$, either via representations of $\pi_1(M)$ into $\mathrm{PSL}(2, \mathbb{R})$ (which in turn acts on $S^1$), or by creating a co-orientable, taut foliation of $M$ and then applying Thurston's universal circle construction to arrive at $\rho : \pi_1(M) \rightarrow \mathrm{Homeo}_+(S^1)$.  For examples of these approaches, see \cite{CD18, BH18}.  In either case the construction of the representation or foliation must be done with great care, so as to guarantee that the corresponding action on $S^1$ has trivial Euler class---thus guaranteeing that $\pi_1(M)$ is left-orderable (cf. Corollary \ref{euler corollary}).

Our result provides an alternative approach to this problem.  One could instead take care to create a family of faithful representations of $\pi_1(M)$ or a family of co-orientable taut foliations of $M$ that depend on a parameter $n \in \mathbb{N}_{>1}$, each yielding a faithful action of $\pi_1(M) \times \ZZ/n\ZZ$ on $S^1$, and pay no heed to the Euler class of these actions. Such a family of actions, together with Theorem \ref{main theorem}, will also yield left-orderability of $\pi_1(M)$.

\subsection{Organization.} The paper is organized as follows.   In Section \ref{background} we review the classical arguments that connect left-orderability of a group to the existence of a circular ordering having certain cohomological properties.  Section \ref{main theorem section} contains the proof of Theorem \ref{main theorem}.  Last, Section \ref{obstruction section} introduces and studies the obstruction spectrum of a group, including the examples which show Theorem \ref{main theorem} is sharp.

\subsection*{Acknowledgements.}  We thank Dave Morris for his helpful conversations, in particular for his help in generalizing the arguments of Section \ref{obstruction section}, and for supplying the statement and proof of Theorem \ref{theorem-of-dave}. We would also like to thank Kathryn Mann for asking whether there existed torsion-free circularly-orderable groups that are not left-orderable, providing the impetus for this project, and Steve Boyer for drawing our attention to the construction mentioned in Remark \ref{not-good-enough} and Example \ref{knot_example}. We are grateful to Sanghoon Kwak for comments on an earlier draft, and to the referee for comments and suggestions to improve the paper.


\section{Background}
\label{background}

A group $G$ is \textit{circularly-orderable} if there exists a function $c:G^3 \rightarrow \{ 0, \pm 1\}$ satisfying:
\begin{enumerate}
\item $c^{-1}(0) = \{ (g_1, g_2, g_3) \mid g_i = g_j \mbox{ for some $i \neq j$ } \}$.
\item $c$ satisfies a cocycle condition, meaning
\[c(g_2, g_3, g_4) - c(g_1, g_3, g_4) + c(g_1, g_2, g_4) - c(g_1, g_2, g_3) = 0
\]
for all $g_1, g_2, g_3, g_4 \in G$.

\item $c(g_1, g_2, g_3) = c(hg_1, hg_2, hg_3)$ for all $h, g_1, g_2, g_3 \in G$.
\end{enumerate}
A function $c:G^3 \rightarrow \{0, \pm 1\}$ satisfying the conditions above will be called a \textit{circular ordering} of $G$.  When $G$ comes equipped with a circular ordering $c$, the pair $(G, c)$ will be called a \textit{circularly-ordered group}.  Note that unless otherwise specified, all circular orderings in this note are assumed to be left-invariant.

Promoting $(G,c)$ from a circularly-ordered to a left-ordered (or left-orderable) group has often focused on conditions on the circular ordering $c$ which will guarantee left-orderability of $G$. The ideas are classical, and proceed as follows.

From every circularly-ordered group $(G, c)$ we can construct a group $\widetilde{G}_c$ via a type of ``unwrapping" construction that produces a left-ordered cyclic central extension \cite{zheleva76}.  The group $\widetilde{G}_c$ is the set $\mathbb{Z} \times G$ equipped with the operation $(n,a)(m,b)=(n+m+f_c(a,b),ab)$, where
\[
f_c(a,b)=\left\{\begin{array}{cl} 0 & \text{if $a=id$ or $b=id$ or $c( id, a, ab)=1$}
\\ 1 &  \text{if $ab=id$ $(a\not=id)$ or $c(id,ab,a)=1$.  }  \end{array}
\right.
\] 
Setting $P=\{(n,a)\mid n\geq 0\} \setminus \{ (0, id) \},$ it is not difficult to check that $P$ is the positive cone of a left ordering $<_c$ of $\widetilde{G}_c$.   

The construction of $\widetilde{G}_c$ is a special case of a classical result:  The function $f_c(a,b)$ is in fact an inhomogeneous $2$-cocycle, and so it defines an element $[f_c] \in H^2(G; \mathbb{Z})$.  The construction above is then the well-known correspondence between elements of $H^2(G; \mathbb{Z})$ and equivalence classes of central extensions 
\[ 1 \rightarrow \mathbb{Z} \rightarrow \widetilde{G} \rightarrow G \rightarrow 1,
\]
applied to the cocycle $f_c$.
Since this correspondence carries the identity in $H^2(G; \mathbb{Z})$ to a split extension, we immediately have the following.

\begin{corollary}
\label{euler corollary}
If $(G, c)$ is a circularly-ordered group and $[f_c]=id \in H^2(G; \mathbb{Z})$, then $G$ is left-orderable.
\end{corollary}
\begin{proof}
If $[f_c] =id$ then the left-ordered cyclic central extension $\widetilde{G}_c$ is isomorphic to $\ZZ \times G$, so $G$ is left-orderable.
\end{proof}

\subsection{Secret left orderings}
In Corollary \ref{euler corollary} it is important to note that $c$ itself does not provide any hint as to what the left ordering of $G$ might be.  There are cases, however, where the provided circular ordering does carry such information, and again this behaviour can be detected using a cohomological argument.  

First we note that every left-ordered group $(G, <)$ provides an example of a circularly-ordered group, by declaring $c(g_1, g_2, g_3) = 1$ whenever $g_1 <g_2<g_3$, up to cyclic permutation.  The circular orderings of a group that arise this way will be called \textit{secret left orderings}, and obviously carry all the information needed to define a left ordering of $G$.  Such circular orderings are detected by the second bounded cohomology group $H^2_b(G;\ZZ)$.

\begin{proposition} \label{bounded-characterization}
Let $(G,c)$ be a circularly-ordered group.  The circular ordering $c$ is a secret left ordering if and only if $[f_c] = id \in H^2_b(G;\ZZ)$.
\end{proposition}
\begin{proof}
Suppose $c$ is a secret left ordering.  Let $P$ be the positive cone of the secret left ordering $<$ corresponding to $c$.  Define a bounded function $d:G \to \ZZ$ by 
\[
d(g) = \begin{cases}
0 & \text{if } g \in P \cup \{id\} \\
1 & \text{if } g \in P^{-1}.
\end{cases}
\]
Then for all $g,h \in G$, $f_c(g,h) = d(g) - d(gh) + d(h)$.  Therefore $[f_c] = id \in H^2_b(G;\ZZ)$.

Conversely, suppose that for all $g,h \in G$, $f_c(g,h) = d(g) - d(gh) + d(h)$ for some bounded function $d:G \to \ZZ$, and recall $f_c(g,h) \in \{0,1\}$.  We first show $d(g) \in \{0,1\}$ for all $g \in G$.

Suppose $d(g) < 0$ for some $g \in G$.  Then $d(g) + d(g^{k-1}) - d(g^k) = f_c(g,g^{k-1}) \geq 0$ and via induction on $k$ we may conclude that $d(g^k) \leq -k$ for all $k \in \NN$, contradicting the boundedness of $d$. Similarly, if $d(g) \geq 2$, induction on $k$ with the fact that $d(g) + d(g^{k-1}) - d(g^k) \leq 1$ allows us to conclude $d(g^k) \geq 1 + k$ for all $k \in \NN$.  Therefore $d(g) \in \{0,1\}$ for all $g \in G$.

Now define $P \subset G$ by $g \in P \cup \{id\}$ if and only if $d(g) = 0$.  We check that $P$ is a positive cone.  Given $g,h \in P$, note $d(g) + d(h) - d(gh) \in \{0,1\}$.  Therefore we must have $d(gh) = 0$ so $P\cdot P \subset P$.  Let $g \in G \sm\{id\}$. If $d(g) = 0$, then $d(g) + d(g^{-1}) = f_c(g,g^{-1}) = 1$.  Therefore $d(g^{-1}) = 1$.  Similarly if $d(g) = 1$, $d(g^{-1}) = 0$.  Therefore $G = P \sqcup \{id\} \sqcup P^{-1}$ and $P$ is a positive cone on $G$.

Let $<$ be the left ordering on $G$ with positive cone $P$.  It suffices to show $g_1<g_2<g_3$ implies $c(g_1,g_2,g_3) = 1$.  If $g_1<g_2<g_3$, then $d(g_1^{-1}g_2) = d(g_1^{-1}g_3) = d(g_2^{-1}g_3) = 0$.  We can recover $c$ from $f_c$ via the prescription
\[
c(g_1,g_2,g_3) = \begin{cases}
0 &\text{if } g_i = g_j \text{ for some }i \neq j \\
1 - 2f_c(g_1^{-1}g_2,g_2^{-1}g_3) &\text{otherwise.}
\end{cases}
\]
Since $g_1,g_2,$ and $g_3$ are distinct we have
\[
c(g_1,g_2,g_3) = 1 - 2f_c(g_1^{-1}g_2,g_2^{-1}g_3) = 1 - 2(d(g_1^{-1}g_2) + d(g_2^{-1}g_3) - d(g_1^{-1}g_3)) = 1
\]
completing the proof.
\end{proof}

\begin{remark}\label{not-good-enough}
The astute reader will notice that if $(G,c)$ is a circularly-ordered group, then $c$ is a bounded homogeneous 2-cocycle. In fact, $[c] = -2[f_c] \in H^2_b(G;\ZZ)$ (and similarly in $H^2(G;\ZZ)$). However, it is worth pointing out that Corollary \ref{euler corollary} and Proposition \ref{bounded-characterization} do not hold if we replace $[f_c]$ with $[c]$. Indeed, the only circular ordering $c$ on $\ZZ/2\ZZ$ is such that $[c] = 0$ in $H^2_b(\ZZ/2\ZZ;\ZZ)$, and therefore in $H^2(\ZZ/2\ZZ;\ZZ)$, but $\ZZ/2\ZZ$ is of course not left-orderable. More generally, if $G$ is a circularly-orderable and not left-orderable group with an index 2 left-orderable subgroup $H$, one can construct a circular ordering $c$ on $G$ such that $[c] = 0$ in $H^2_b(G;\mathbb Z)$.
\end{remark}

\subsection{Lexicographic orderings} \label{lex-orderings}
The following construction intertwines left-orderable and circularly-orderable groups, and is a standard construction (see \cite[Lemma 2.2.12]{Calegari06}).

Suppose 
\[
1 \lra K \lra G\overset{\phi}{\lra} H \lra 1
\]
is a short exact sequence such that $c_H$ is a circular ordering on $H$ and $<$ is a left-ordering on $K$. Let $c_<$ be the corresponding secret left ordering of $K$, that is, $c_<(g_1,g_2,g_3) = 1$ whenever $g_1 < g_2 <g_3$. Define the \textit{lexicographic circular ordering} $c$ on $G$ by
\[
c(g_1,g_2,g_3) = \begin{cases}
c_H(\phi(g_1),\phi(g_2),\phi(g_3)) &\text{if $\phi(g_1),\phi(g_2)$, and $\phi(g_3)$ are all distinct},\\
c_<(g_2^{-1}g_1,id,g_1^{-1}g_2) &\text{if $\phi(g_1) = \phi(g_2) \neq \phi(g_3)$},\\
c_<(g_1^{-1}g_3,id,g_1^{-1}g_2) &\text{if $\phi(g_1) = \phi(g_2) = \phi(g_3)$}.
\end{cases}
\]
Note that circular orderings are invariant under cyclic permutations of the arguments, so the information provided uniquely determines the lexicographic circular ordering.

\section{Promoting circularly-orderable groups}
\label{main theorem section}

Our approach to the problem of promoting the circularly-ordered group $(G, c)$ to a left-orderable group will be to impose conditions on the group $G$, as opposed to conditions on the circular ordering $c$. One obvious obstruction to left-orderability of a circularly-ordered group $G$ is that the group may admit torsion, and so as a first attempt one might ask if removing this obstruction is sufficient to obtain left-orderability.  

If $c$ happens to be a bi-invariant circular ordering (meaning that $c(g_1, g_2, g_3) = c(g_1h, g_2h, g_3h)$ as well as $c(g_1, g_2, g_3) = c(hg_1, hg_2, hg_3)$ for all $h, g_1, g_2, g_3 \in G$), then it turns out that this is enough to guarantee that $G$ is bi-orderable.  To prove this we employ a theorem of \'{S}wierczkowski:

\begin{theorem}
\label{sw theorem}
 \cite{Sw59} If $G$ admits a bi-invariant circular ordering then there is a bi-ordered group $H$ and an order-embedding $i: G \rightarrow H \times S^1$, where $S^1$ is equipped with its natural circular ordering and $H \times S^1$ is ordered lexicographically.
\end{theorem}

\begin{proposition}
A group is bi-orderable if and only if it admits a bi-invariant circular ordering and is torsion free.
\end{proposition}
\begin{proof}
If $G$ admits a bi-ordering $<$, then $G$ is clearly torsion free and setting $c(g_1, g_2, g_3) = 1$ whenever $g_1 <g_2 <g_3$ suffices to determine a bi-invariant circular ordering $c$ of $G$.

On the other hand, if $G$ admits a bi-invariant circular ordering then by Theorem \ref{sw theorem} there exists a bi-ordered group $H$ and an order-embedding $i: G \rightarrow H \times S^1$.  Let $T$ denote the subgroup $\mathbb{Q}/\mathbb{Z}$ of $S^1$.   Assuming that $G$ is torsion-free then yields an embedding of $G$ into $H \times S^1/T$.  The group $H \times S^1/T$ is bi-orderable since it is the product of a bi-orderable group and a torsion-free abelian group, thus so is $G$.
\end{proof}

In contrast, imposing torsion-freeness on a group admitting a left-invariant circular ordering is not enough to obtain left-orderability.  Torsion-free circularly-orderable groups that are not left-orderable are relatively easy to find among the class of $3$-manifold groups.

\begin{example}\label{seifert_example}
Let $M$ be a Seifert fibred rational homology sphere of the form $S^2(\frac{p_1}{q_1}, \frac{p_2}{q_2}, \frac{p_3}{q_3})$ where $\frac{1}{q_1}+ \frac{1}{q_2}+ \frac{1}{q_3} < 1$, where $p_i$ and $q_i$ are relatively prime for $i=1, 2, 3$ and represent the (unnormalized) Seifert invariants (see \cite{BRW05} or \cite{NJ} for the relevant background).  Then the fundamental group of $M$ is torsion-free, and moreover it fits into a short exact sequence
\[ 1 \rightarrow \mathbb{Z} \rightarrow \pi_1(M) \rightarrow \Delta \rightarrow 1
\]
where $\mathbb{Z}$ is central and generated by the class of the fibre of $M$, and $\Delta$ is a $(q_1, q_2, q_3)$-triangle group which is isomorphic to a subgroup of $\mathrm{PSL}(2, \mathbb{R})$ since  $\frac{1}{q_1}+ \frac{1}{q_2}+ \frac{1}{q_3} < 1$.  Since $\mathrm{PSL}(2, \mathbb{R})$ acts faithfully in an orientation-preserving way on $S^1$, it is circularly-orderable, thus so is $\Delta$.  Therefore $\pi_1(M)$ can be circularly ordered via a lexicographic argument.

On the other hand, $\pi_1(M)$ is left-orderable if and only if $M$ admits a horizontal foliation  \cite{BRW05}.  This happens exactly when the normalized Seifert invariants and Euler number of the Seifert fibration satisfy a certain system of diophantine inequalities \cite{EHN81, JN85, Naimi94}. 
\end{example}

\begin{example}\label{knot_example}
Let $K$ be a fibred hyperbolic knot in $S^3$ for which $\frac{p}{q}$-surgery yields a manifold $K(\frac{p}{q})$ with non-left-orderable fundamental group whenever $\frac{p}{q} \geq r$ for some $r \in \RR$.  Examples of such knots are provided by \cite{CW13, CGH16} and others, but for the sake of concreteness we may take $K$ to be the $(-2, 3, 2n+1)$-pretzel knot where $n \geq 3$, in which case $r = 2n+3$ by \cite{Nie18}.  

The groups $\pi_1(K(\frac{p}{q}))$ for $\frac{p}{q} \geq r$, despite being non-left-orderable, are generically circularly-orderable.  If $p \geq 2$, by \cite[Corollary 1.5]{BH18} the kernel of the short exact sequence
\[ 1 \rightarrow H \rightarrow \pi_1(K(p/q)) \rightarrow \mathbb{Z}/p\mathbb{Z} \rightarrow 1
\] 
is left-orderable for all but at most two values of $q$ coprime to $p$, and thus $ \pi_1(K(\frac{p}{q}))$ is circularly-orderable by a lexicographic argument.  Moreover, since $K$ is a hyperbolic knot \cite{KL}, $K(\frac{p}{q})$ is hyperbolic for all but finitely many $\frac pq$ \cite{Thurston}, and therefore all but finitely many of the groups $\pi_1(K(\frac{p}{q}))$ are torsion-free.
\end{example}

\subsection{Preliminary results}
The results in this section are essential to the proofs of Theorems \ref{tararin's theorem}, \ref{LO criterion}, and \ref{theorem-of-dave}

Let $G$ be a left-ordered group with ordering $<$.  A subgroup $C$ of $G$ is said to be \textit{$<$-convex} if for every $g, h \in C$ and $f \in G$, the implication $g<f<h \Rightarrow f \in C$ holds.  A subgroup $C$ of a left-orderable group $G$ is called \textit{relatively convex} if there exists a left ordering $<$ of $G$ such that $C$ is $<$-convex. The next two propositions are well known, and the proofs are left to the reader.

\begin{proposition}\label{convex-characterisation}
Let $(G, <)$ be a left-ordered group and $C$ a subgroup of $G$.  Then $C$ is $<$-convex if and only if the ordering $\prec$ of the left cosets $\{ gC\}_{g \in G}$ defined by $gC \prec hC \Leftrightarrow g<h$ is well-defined.
\end{proposition}

It follows that a normal subgroup $C$ of a left-orderable group $G$ is relatively convex if and only if the quotient $G/C$ is left-orderable.  

\begin{proposition}\label{relatively-convex-transitive}
Let $G$ be a left-orderable group and $H$ a relatively convex subgroup. If $K$ is a relatively convex subgroup of $H$, then $K$ is a relatively convex subgroup of $G$.
\end{proposition}

The next lemma is a general result about convex subgroups containing normal subgroups of a left-ordered group.

\begin{lemma}
\label{normal-convex-closure}
Let $G$ be a group with left-ordering $<$, and suppose $H \triangleleft G$ is a normal subgroup. The set 
\[
T= \{g \in G \mid \text{there exists } h_1,h_2 \in H \text{ such that } h_1 < g < h_2\}
\]
is a $<$-convex subgroup of $G$.
\end{lemma}
\begin{proof}
To ease notation, for $a \in T$ and $h \in H$, let $h^a = a^{-1}ha \in H$. Suppose $a,b \in T$, so there exists $h_1,h_2,k_1,k_2 \in H$ such that $h_1 < a < h_2$ and $k_1 < b < k_2$. Then $ab < ak_2 = k_2^aa < k_2^ah_2$ and $ab > ak_1 = k_1^aa > k_1^ah_1$. Since $k_i^ah_i \in H$ for $i \in \{1,2\}$, we conclude that $ab \in T$. To see $a^{-1} \in T$, we have $id > a^{-1}h_1 = h_1^{a^{-1}}a^{-1}$ so $(h_1^{a^{-1}})^{-1} > a^{-1}$. Similarly, $(h_2^{a^{-1}})^{-1} < a^{-1}$, so $a^{-1}\in T$ and $T$ is a subgroup of $G$. For convexity, suppose there exists a $g \in G$ such that $a < g < b$, where $a,b \in T$ are as above. Then $h_1 < a < g < b < k_2$, so $g \in T$, completing the proof.
\end{proof}

For the remainder of this section, let $Q$ be a subgroup of $(\QQ,+)$. Let $G$ be a left-orderable group with a normal subgroup $R$ isomorphic to $Q$. Let $K$ be the intersection of all relatively convex subgroups of $G$ containing $R$, which is again a relatively convex subgroup of $G$ (see for example \cite[Proposition 2.18]{CR16}).

Recall that in a left-ordered group $(G, <)$ a set $X \subset G$ is called \textit{$<$-cofinal} if for every $g \in G$ there exists $x, y \in X$ such that $x<g<y$.  An element $g \in G$ is called $<$-cofinal if $\langle g \rangle$ is cofinal as a set.

\begin{lemma} \label{cofinal-central-subgroup}
Let $R \subset K \subset G$ be given as above. Then $R$ is $<$-cofinal for all left orderings $<$ of $K$, and $R$ is central in $K$.
\end{lemma}
\begin{proof}
Choose an arbitrary left ordering $<$ of $K$. If $R$ is bounded above in the left-ordered group $(K,<)$, then by Lemma \ref{normal-convex-closure} there would be a proper $<$-convex subgroup in $K$ containing $R$, which by Proposition \ref{relatively-convex-transitive} contradicts the definition of $K$. Therefore $R$ is $<$-cofinal for all left orderings $<$ of $K$.

Fix a left-ordering $<$ of $K$. Since $R$ is $<$-cofinal, for each positive $a \in K$ and positive $q \in R$, there exists a $k$ such that $a< q^k$. It follows that $id < a^{-1}q^ka$. Therefore $a^{-1}qa$ is positive in $R$. Let $P$ be the positive cone of $(R,<)$. We now have that for any positive element $a \in K$, $a^{-1}Pa \subset P$. However, since $a^{-1}Pa$ and $P$ are both positive cones of left orderings on $K$, we can conclude $a^{-1}Pa = P$. Therefore $P = aPa^{-1}$ and therefore $k^{-1}Pk = P$ for all $k \in K$.

Conjugation of $R$ by elements of $K$ now gives a well-defined homomorphism $\phi:K \to \operatorname{Aut}(R,<)$, whose kernel is the centralizer $C$ of $R$. Since $\operatorname{Aut}(H,<)$ is a left-orderable group, $C$ is relatively convex in $K$, so we must have that $C = K$. Therefore $R$ is central in $K$.
\end{proof}

Every subgroup $Q$ of $(\QQ,+)$ has exactly two left orderings, which are the restrictions of the two left orderings on $\QQ$. A subgroup $Q \subset \QQ$ is \textit{dense} if for all $a,b \in Q$ such that $a < b$, there exists a $c \in Q$ such that $a < c < b$, where $<$ is any of the two left orderings on $Q$. Equivalently, $Q \subset \QQ$ is a dense subgroup if and only if $Q \not\cong \ZZ$.

\begin{lemma} \label{normal-infinitesimals}
Let $R \subset K \subset G$ be given as above, and suppose $R$ is isomorphic to a dense subgroup of $\QQ$. Let $<$ be a left ordering of $K$ and consider the set of \textit{infinitesimals} $M \subset K$ given by
\[
M = \{k \in K \mid p^{-1} < k < q \mbox{ for all } p, q > id, p, q \in R \}.
\]
Then 
\begin{enumerate}
\item $M$ is a normal subgroup of $K$ for all left orderings $<$ of $K$, and
\item there exists a left ordering $\prec$ of $K$ such that $K/M$ is isomorphic to a subgroup of $(\QQ,+)$.
\end{enumerate}
\end{lemma}
\begin{proof}
For 1, fix an arbitrary left ordering $<$ of $K$. By Lemma \ref{cofinal-central-subgroup}, $R$ is both central and $<$-cofinal in $K$. It follows from the facts that $R$ is central and isomorphic to a dense subgroup of $\QQ$ that $M$ is a subgroup. Furthermore, since $R$ is $<$-cofinal, $M$ is precisely the set of elements in $K$ that are not $<$-cofinal. Thus to show that $M$ is normal, it suffices to show that for any $k \in K$, $k$ is $<$-cofinal if and only if $fkf^{-1}$ is $<$-cofinal for all $f \in K$.

Suppose that $k$ is $<$-cofinal. We deal only with the case of $k>id$, the case of $k<id$ being similar.  Then there exists $p, q \in K$ with $p,q >id$ and $t \in \mathbb{N}$ such that $ p<k^t<q$.

Now let $f \in K$ be given.  We can choose $a, b \in R$ such that $a<f<b$, note that $a$ and $b$ can be chosen so that $pab^{-1} >id$.  Now the inequalities $p<k^t$, $a<f$ and $b^{-1}<f^{-1}$ combine to give $pab^{-1} <fk^tf^{-1}$ since $R$ is central.  Note that since $id<pab^{-1}$ is cofinal and $(pab^{-1})^n < (fk^tf^{-1})^n = (fkf^{-1})^{tn}$ for all $n$, it follows that $fkf^{-1}$ is cofinal as well. This completes the proof of 1.

For 2, consider the set $\mathcal R$ of relatively convex subgroups $L$ of $K$ other than $K$ itself. Since $R$ is $<$-cofinal for all left-orderings $<$ of $K$, $\mathcal R$ is the set of all relatively convex subgroups $L$ of $K$ with the property that $L \cap R = \{id\}$. The set $\mathcal R$ forms a poset under inclusion. Since the union of any chain of relatively convex subgroups is again relatively convex \cite{Kokorin}, every chain has an upper bound. Therefore by Zorn's lemma, there exists a maximal element in this poset, call it $N$.

Fix a left-ordering $\prec$ of $K$ such that $N$ is $\prec$-convex. Since $H$ is $\prec$-cofinal and $N \cap H = \{id\}$, it must be the case that
\[
N = \{ k \in K \mid p^{-1} \prec k \prec q \mbox{ for all } p, q \succ id, p, q \in R\}.
\]
The set $N$ is in fact the infinitesimals, so $N$ is a normal subgroup of $K$ by statement 1 of the theorem. Consider the quotient $K/N$ and let $\pi:K \to K/N$ be the quotient map. Since $N$ is relatively convex, $K/N$ is left-orderable. Moreover, $K/N$ contains no proper relatively convex subgroups, since the preimage of such a group under $\pi$ would be a proper relatively convex subgroup containing $N$. Therefore $K/N$ is a torsion-free rank one abelian group \cite[Proposition 5.1.9]{KM96}, and therefore $K/N$ is isomorphic to a subgroup of $\QQ$.
\end{proof}

\subsection{The main theorem.}
Theorem \ref{tararin's theorem} is a result of Tararin, which plays a key role in the proof of our main criterion.  We therefore provide a proof for the sake of completeness.  

\begin{theorem}\cite[Theorem 5.4.1]{KM96}
\label{tararin's theorem}
If a normal subgroup of a left-orderable group $G$ is isomorphic to $\mathbb{Q}$, then it is a relatively convex subgroup of $G$.
\end{theorem}
\begin{proof}
Let $R$ denote a normal subgroup of $G$ isomorphic to $(\QQ,+)$ and let $K$ be the intersection of all relatively convex subgroups of $G$ that contain $R$. Since the intersection of relatively convex subgroups is relatively convex (see for example \cite[Proposition 2.18]{CR16}), $K$ is relatively convex. Our goal is to show $K = R$. To that end, let $M$ be a normal subgroup of $K$ such that $M \cap R = \{id\}$ and $K/M$ is isomorphic to a subgroup of $\QQ$. Such a subgroup exists by Lemma \ref{normal-infinitesimals}.

Let $\pi:K \to K/M$ be the quotient map. Since $M \cap R = \{id\}$, $R$ maps injectively into $K/M$, so $\pi(R)$ is a subgroup of $K/M$ isomorphic to $\QQ$. Therefore $\pi(R) = K/M$ and $K = MR$. Since $R$ is central by Lemma \ref{cofinal-central-subgroup}, we may now conclude that $K \cong M \times R$, so $R$ is relatively convex in $K$. Since $K$ is the intersection of all relatively convex subgroups of $G$ containing $R$, by Proposition \ref{relatively-convex-transitive} we conclude that $M \cong \{id\}$ and $K = R$.  
\end{proof}

The next proposition is stated in more general terms than what is needed in this section, but its full strength will be used later in the proof of Theorem \ref{theorem-of-dave}. Recall from Section \ref{background} that for a circularly-ordered group $(G,c)$, there exists a left-orderable central extension $\wt{G}_c$. 

\begin{proposition}
\label{divisible lift}
Let $Q$ be a subgroup of $\QQ$ containing $\ZZ$. For every circular ordering $c$ of $Q/\mathbb{Z}$, the lift $\wt{(Q/\mathbb{Z})}_c$ is isomorphic to $Q$.
\end{proposition}
\begin{proof}  In fact we will prove something stronger.  Suppose $1 \to \ZZ \overset{\iota}\to G \to Q/\ZZ \to 1$ is a short exact sequence with $G$ torsion-free.  We will prove $G \cong Q$.

Let $H \subset G$ be a finitely generated subgroup.  Since $H$ is torsion-free, $\iota(\ZZ) \cap H \cong \ZZ$. Since every finitely-generated subgroup of $Q/\ZZ$ is a finite cyclic group, restricting the short exact sequence to $H$ gives $1 \to \ZZ \to H \to \ZZ/k\ZZ \to 1$ for some $k \in \NN$.  Now $H$ is a torsion-free group with a finite-index subgroup isomorphic to $\ZZ$, so $H \cong \ZZ$.  Therefore $G$ is a torsion-free locally cyclic group, so it is isomorphic to a subgroup of $\QQ$.  Fix an identification $G \subset \QQ$ with the copy of $\ZZ$ corresponding to the kernel of the map $G\to Q/\ZZ$ being precisely the subgroup $\mathbb{Z}$ of $\mathbb{Q}$.

Let $\ZZ \subset R \subset \QQ$ be an arbitrary subgroup containing $\ZZ$, and let $\pi:\QQ \to \QQ/\ZZ$ be the projection map. We claim that $\pi(R)$ contains the unique subgroup of order $k$ in $\QQ/\ZZ$ if and only if $\langle \frac1k \rangle$ is a subgroup of $R$. If $\langle \frac1k \rangle \subset R$, then it is clear that $\pi(R)$ contains the unique subgroup of order $k$. For the other direction, suppose there is an $r \in R$ such that $\pi(r)$ has order $k$. Then $r = \frac ak$ for some $a \in \ZZ$ such that $\gcd(a,k) = 1$. Therefore there exists integers $x$ and $y$ such that $ax + ky = 1$. Then $xr + y = \frac1k(ax + ky) = \frac1k$, which is in $R$ since $\ZZ \subset R$ and $\langle \frac 1k \rangle \subset R$.

Now we have subgroups $G$ and $Q$ of $\QQ$ such that $\ZZ \subset G,Q \subset \QQ$ and $\pi(G) = \pi(Q)$. Let $0 \neq g = \frac ak$ where $a$ and $k$ are coprime. Then $\pi(g)$ has order $k$, so $\langle \frac 1k \rangle$ is a subgroup of both $G$ and $Q$, and since $g \in \langle \frac 1k \rangle$, we have $G \subset Q$. The converse is the same, so $G = Q$.
\end{proof}

\begin{theorem}
Suppose that $G$ is a circularly-orderable group and that $H$ is a normal subgroup of $G$ isomorphic to $\mathbb{Q} /\mathbb{Z}$.  Then $G/H$ is left-orderable.
\end{theorem}
\begin{proof}
Fix a circular ordering $c$ of $G$, which we will use to construct lifts.  Since $H$ is normal in $G$, it follows that $\wt H$ is normal in $\wt G$; since $H$ is isomorphic to $\mathbb{Q} /\mathbb{Z}$ we know that $\wt H \cong \mathbb{Q}$ by Proposition \ref{divisible lift}.  Thus $\wt G / \wt H$ is left-orderable by Proposition \ref{convex-characterisation} and Theorem \ref{tararin's theorem}, so the proof is complete once we observe that the map $\phi: \wt G / \wt H \rightarrow G/H$ given by $(n, g)\wt H \mapsto gH$ is an isomorphism.

The map $\phi$ is well-defined, because if $(n, g) \wt H = (m, h) \wt H$ then there exists $(\ell, f) \in \wt H$ (in particular, $f \in H$) such that $$(n,g) = (m, h)(\ell , f) = (m+\ell+f_c(h, f), hf).$$  Thus $g=hf$ and so $gH = hH$.

The map is injective because if $gH = hH$ then $g=hf$ for some $f \in H$.  Then for every $n, m \in \mathbb{Z}$ the quantity $\ell = n-m-f_c(h,f)$ yields $(\ell, f) \in \wt H$ with $(n,g) = (m,h)(\ell, f)$, so that $(n,g) \wt H = (m,h) \wt H$.  That the map is surjective is obvious.
\end{proof}

\begin{corollary}
\label{LO criterion}
A group $G$ is left-orderable if and only if $G \times \mathbb{Q}/\mathbb{Z}$ is circularly-orderable.
\end{corollary}
\begin{proof}
If $G \times \mathbb{Q}/\mathbb{Z}$ is circularly-orderable then the previous theorem implies left-orderability of $G$.  On the other hand, if $G$ is left-orderable then $G \times \mathbb{Q}/\mathbb{Z}$ is circularly-orderable by the lexicographic construction from Section \ref{lex-orderings}.
\end{proof}

\begin{corollary}\label{spectrum existence}
A group $G$ is left-orderable if and only if $G \times \mathbb{Z}/ n \mathbb{Z}$ is circularly-orderable for all $n \geq 2$.
\end{corollary}
\begin{proof}
Assume that $G \times \mathbb{Z}/ n \mathbb{Z}$ is circularly-orderable for all $n \geq 2$. Choose a finitely generated subgroup $H$ of $G \times \mathbb{Q}/\mathbb{Z}$, and let $\pi_1$ and $\pi_2$ be the projections from $G \times \mathbb{Q}/\mathbb{Z}$ onto the first and second factors respectively.  Then $H  \subset \pi_1(H) \times \pi_2(H) \subset G \times \mathbb{Z} /n \mathbb{Z}$ for some $n$, since every finitely generated subgroup of $\mathbb{Q}/\mathbb{Z}$ is cyclic.  As $G \times \mathbb{Z} /n \mathbb{Z}$ is assumed to be circularly-orderable, so is $H$; as $H$ was an arbitrary finitely generated subgroup of $G \times \mathbb{Q}/\mathbb{Z}$ we can conclude that $G \times \mathbb{Q}/\mathbb{Z}$ is circularly-orderable \cite[Lemma 2.14]{Clay}.   It follows that $G$ is left-orderable.

On the other hand, if $G$ is left-orderable then circular-orderability of $G \times \mathbb{Z}/ n \mathbb{Z}$ for all $n$ is again the lexicographic construction from Section \ref{lex-orderings}.
\end{proof}

\section{The obstruction spectrum of a circularly-orderable group}
\label{obstruction section}

In light of Corollary \ref{spectrum existence}, for any circularly-orderable group $G$ we can define the {\it obstruction spectrum of $G$} as
\[
\operatorname{Ob}(G):= \{n \in \mathbb N_{> 1} \mid G \times \mathbb Z/n\mathbb Z \mbox{ is not circularly-orderable}\}.
\]
With this notation, Corollary \ref{spectrum existence} is equivalent to the statement that a circularly-orderable group $G$ is left-orderable if and only if $\operatorname{Ob}(G) = \emptyset$.  It is worth noting that if $n \mid m$, and $n \in \Ob(G)$ for some circularly-orderable group $G$, then $m \in \Ob(G)$.  

In the remaining sections we investigate what properties of a group $G$ are captured by $\Ob(G)$, and what possible subsets of $\mathbb{N}_{>1}$ occur as $\Ob(G)$ for some non-left-orderable group $G$.

\subsection{Properties of a group encoded by its obstruction spectrum.} 
\label{obstruction structure subsection}
Since a finite group is circularly-orderable if and only if it is cyclic, the obstruction spectrum detects torsion.  Indeed, suppose there is an element of order $k$ in a circularly-ordered group $G$.  Then $\operatorname{Ob}(G) \supset \bigcup_{p \mid k} p\NN$ where $p$ is a prime number.  More explicitly, we set
\[
T(G) = \{ k \in \mathbb{N} \mid \exists g \in G \mbox{ of order $k$} \}
\]
and define 
\[ 
\operatorname{Ob}_T(G):= \{ n \in \mathbb{N}_{>1} \mid \exists k \in T(G) \mbox{ such that } \gcd(k, n) \neq 1 \}.
\]
Note that $\operatorname{Ob}_T(G) \subset \operatorname{Ob}(G)$, because if $g \in G$ is of order $k$ and $\gcd(k,n) \neq 1$ then $\langle g \rangle \times \mathbb{Z} / n \mathbb{Z}$ is not cyclic, hence not circularly-orderable.  We call $\operatorname{Ob}_T(G)$ the \textit{torsion part} of the obstruction spectrum.

In the case that $G$ admits a bi-invariant circular ordering, the obstruction spectrum of $G$ is exactly equal to its torsion part.  

\begin{proposition}
If $G$ admits a bi-invariant circular ordering then $\Ob(G) = \Ob_T(G)$.
\end{proposition}
\begin{proof}
Suppose $G$ admits a bi-invariant circular ordering, and let $i : G \rightarrow H \times S^1$ be an order-embedding, where $H$ is a bi-ordered group and $H \times S^1$ is equipped with a lexicographic bi-invariant circular ordering (Theorem \ref{sw theorem}).  Then if $g \in G$ is of order $k$, $i(g)=(id, r)$ where $r$ is a $k$-th primitive root of unity.  If $n>1$ and $(k, n) =1$ for every $k \in T(G)$, then for each $n$-th primitive root of unity $r$ we have  $G \times \mathbb{Z}/n\mathbb{Z} \cong \langle i(G), r \rangle \subset H \times S^1$.  Thus $G \times \mathbb{Z}/n\mathbb{Z}$ is circularly-orderable.  The other direction is obvious.
\end{proof}

In general, however, one expects the containment $\Ob_T(G) \subset \Ob(G)$ to be proper, which prompts the following question. 
\begin{question}
If $G$ is torsion-free, what property of $G$ is $\Ob(G)$ detecting?\footnote{We suspect that it may be related to the notion of dynamical forcing of rotation numbers, introduced in \cite{Calegari06}.}
 \end{question}
 In the cases where $G$ is torsion-free and $G$ arises as the fundamental group of  $3$-manifold, such as Examples \ref{seifert_example} and \ref{knot_example}, it would be interesting to explicitly compute and give a topological interpretation of their obstruction spectra.  We expect such a topological interpretation may be possible, since we can already observe that the structure of  $\Ob(G)$ carries topological consequences.  First we note:

\begin{lemma}
\label{mapping lemma}
If $\phi :G \rightarrow H$ is a homomorphism between circularly-orderable groups with left-orderable kernel, then $\Ob(G) \subset \Ob(H)$.
\end{lemma}
\begin{proof}
Suppose that $H \times \ZZ/n\ZZ$ is circularly-orderable for some $n \in \mathbb{N}_{>1}$.  Then the kernel of the map $\psi: G \times \ZZ/n\ZZ \rightarrow H \times \ZZ/n\ZZ$ given by $\psi(g, n) = (\phi(g), n)$ is isomorphic to $\ker(\phi)$, and thus is left-orderable.  Therefore $G \times \ZZ/n\ZZ$ can be lexicographically circularly ordered.
\end{proof}

From this, we are able to prove a generalization of a result of Rolfsen, wherein he used left-orderability of fundamental groups to provide an obstruction to the existence of degree one maps between $3$-manifolds \cite{Rolfsen04}.

\begin{proposition}
\label{manifold obstruction}
Suppose that $M$, $N$ are compact, connected irreducible $3$-manifolds with infinite circularly-orderable fundamental groups.  If there exists a homomorphism $\pi_1(M) \rightarrow \pi_1(N)$ with nontrivial image, then $\Ob(\pi_1(M)) \subset \Ob(\pi_1(N))$.
\end{proposition}
\begin{proof}
Under the assumptions of the proposition, $\pi_1(M)$ and $\pi_1(N)$ are torsion free.  Thus if the image of $\pi_1(M)$ is nontrivial, it is infinite, hence the kernel of the map $\pi_1(M) \rightarrow \pi_1(N)$ is an infinite index subgroup.  Any infinite index subgroup of $\pi_1(M)$ is locally indicable, hence left-orderable \cite[cf. proof of Theorem 1.1]{BRW05}.  The conclusion follows from Lemma \ref{mapping lemma}.
\end{proof}

\begin{corollary} \cite[cf. Theorem 1.1]{Rolfsen04}
With $M$ and $N$ as in Proposition \ref{manifold obstruction}, if $\Ob(\pi_1(M)) \not\subset \Ob(\pi_1(N))$ then every map $f: M \rightarrow N$ is nullhomotopic.
\end{corollary}

\subsection{The structure of obstruction spectra.}
In this section we characterize the sets $S \subset \mathbb{N}_{>1}$ that arise as $\Ob_T(G)$ for some group $G$, and show that every such set also arises as $\Ob(G)$ for some torsion-free group $G$.  

\begin{proposition}
Let $\Pi$ denote the set of prime numbers in $\NN$. 
\begin{enumerate} 
\item For every group $G$ there exists $S \subset \Pi$ such that $\Ob_T(G) = \bigcup_{p \in S} p\NN$.  
\item For every set $S \subset \Pi$ there exists a group $G$ such that $\Ob_T(G) = \bigcup_{p \in S} p\NN$.
\end{enumerate}
\end{proposition}
\begin{proof}
For (1), set $S = \{ p \in \Pi \mid \exists k \in T(G) \mbox{ such that } p | k \}$.  From the remarks at the beginning of Subsection \ref{obstruction structure subsection} it follows that $p \mathbb{N} \subset \Ob_T(G)$ for all $p \in S$.  On the other hand if $p$ is prime and does not divide $k$ for all $k \in T(G)$, then $p \notin \Ob_T(G)$ by definition.

To prove (2), define the subgroup $G_S \subset \QQ/\ZZ$ generated by the set $\{\frac 1p \mid p \in S\}$.  Since $G_S$ contains an element of order $p$ for all $p \in S$, $\bigcup_{p \in S} p \NN \subset \Ob_T(G_S)$.  Conversely, if $\gcd(k,p) = 1$ for all $p \in S$, then $G_S \times \ZZ/k\ZZ \cong H$ where $H \subset \QQ/\ZZ$ is the subgroup generated by $G_S \cup \{\frac 1k\}$.  Therefore $G_S \times \ZZ/k\ZZ$ is circularly-orderable, completing the proof.
\end{proof}

Thus our task is reduced to showing that for all $S\subset \Pi$, there exists a torsion-free circularly-orderable group $G$ such that $\Ob(G) = \bigcup_{p \in S}p\NN$.  We begin by producing a torsion-free group $G$ for each prime $p \geq 2$ for which $\Ob(G) = p \mathbb{N}$.  Our primary tool in the examples that follow is \cite[Proposition 6.17]{Ghys01}, which for finitely generated amenable groups guarantees that all circular orderings of the group arise lexicographically from a short exact sequence $ 1  \rightarrow K \rightarrow G \rightarrow H \rightarrow 1 $ where $K$ is left-orderable and $H$ is a subgroup of $S^1$ (see Proposition \ref{form of circular orders} below for details).  As the quotient map in every such short exact sequence factors thought $G/G'$, our computations of $\Ob(G)$ are tied to the structure of $G/G'$.  There is no reason to expect a similar relationship between $\Ob(G)$ and $G/G'$ to hold for non-amenable groups.

We first prepare several lemmas.

\begin{lemma} 
\label{kernel}
Assume that $G_1, G_2$ and $H$ are groups, and let $\phi : G_1 \times G_2 \rightarrow H$ be a homomorphism with $\phi_i = \phi|_{G_i}$.
If $\phi_2$ is injective, then $\ker \phi \cong \phi_1^{-1}(\phi_2(G_2))$.  
\end{lemma}
\begin{proof}
First note that if $\phi_2$ is injective, then the projection $\pi_1 : G_1 \times G_2 \rightarrow G_1$ is injective on $\ker \phi$.  To see this, note that if $(g, a), (h, b) \in \ker \phi$ and $\pi_1(g, a) =\pi_1(h, b)$ then $g=h$ so that $(g,a)^{-1}(h,b) = (1, a^{-1}b) \in \ker \phi$, and so $a=b$ by injectivity of $\phi_2$.  Thus $\ker \phi \cong \pi_1(\ker \phi)$.

We conclude by observing that $\phi_1^{-1}(\phi_2(G_2)) = \pi_1(\ker \phi)$.  For $g \in \pi_1(\ker \phi)$ happens if and only if there exists $h \in G_2$ such that $(g,h) \in \ker \phi$, which happens if and only if $g \in \phi_1^{-1}(\phi_2(h))$.
\end{proof}

\begin{lemma}
\label{form of circular orders}
Suppose that $(G,c)$ is a countable, amenable circularly-ordered group.  Then $c$ is lexicographic relative to a short exact sequence 
\[ 1 \rightarrow K \rightarrow G \rightarrow C \rightarrow 1
\] 
where $K$ is left-orderable and $C$ is a subgroup of $S^1$.
\end{lemma}
\begin{proof}
Let $\rho_c: G  \rightarrow \mathrm{Homeo_+}(S^1)$ denote the dynamical realization of the circular ordering, and consider $\mathrm{rot} \circ \rho_c$, where $\mathrm{rot}$ is the rotation number.  Since $G$ is amenable, $\mathrm{rot} : \rho_c(G) \rightarrow S^1$ is a homomorphism; it is easy to see that this homomorphism must be circular-order preserving by construction of the dynamical realization.   Moreover, the kernel $K$ of the rotation number homomorphism acts on $S^1$ with a global fixed point and is therefore left-orderable \cite[Proposition 6.17]{Ghys01}.
\end{proof}

\begin{proposition}
\label{daves reduction}
Suppose that $G$ is a countable amenable group.  Then $G \times \ZZ/n\ZZ$ is circularly-orderable if and only if there exists a homomorphism $\phi : G \rightarrow C$ onto a cyclic group $C$ containing a subgroup $H$ of order $n$ such that $\phi^{-1}(H)$ is left-orderable.
\end{proposition}
\begin{proof} 
If $G \times \ZZ/n\ZZ$ is circularly-orderable, then by Lemma \ref{form of circular orders} there exists a subgroup $C$ of $S^1$ and a homomorphism $\psi: G \times \ZZ/n\ZZ \rightarrow C$ with left-orderable kernel $K$.  Let $\phi = \psi|_G$.  Since $K$ is left-orderable the restriction of $\psi$ to $\ZZ/n\ZZ$ must be injective (otherwise $K$ would contain torsion), and so by Lemma \ref{kernel} $K = \ker \psi \cong \phi^{-1}(\psi(\ZZ/n\ZZ))$.  

On the other hand, suppose $G$ admits a homomorphism $\phi$ as in the statement of the theorem and let $\iota : \ZZ/n\ZZ \rightarrow C$ be an embedding of $\ZZ/n\ZZ$.  Define $\psi : G \times \ZZ/n\ZZ \rightarrow C$ by $\psi(g,a) = \phi(g) + \iota(a)$.  Then $\psi$ is injective on restriction to $\ZZ/n\ZZ$, so by Lemma \ref{kernel} we have $\ker \psi \cong \phi^{-1}(\iota(\ZZ/n\ZZ))$, which is assumed to be left-orderable.  Thus $G \times \ZZ/n\ZZ$ is circularly-orderable by a standard lexicographic construction.
 \end{proof}

\begin{proposition}
\label{prop: Adamsreduction}
Suppose that $G$ is a finitely generated amenable circularly-orderable group, and that $G/G'$ is finite.  If $e$ denotes the exponent of $G/G'$, then $e \in \Ob(G)$.
\end{proposition}
\begin{proof}
Suppose that  $\phi: G \rightarrow C$ is a homomorphism onto a cyclic group containing a subgroup of order $e$ with left-orderable preimage, as in Proposition \ref{daves reduction}.  Then as $|C|$ must divide $e$, we conclude $|C| =e$.  But then $\phi^{-1}(C)=G$ is not left-orderable, since every finitely generated amenable left-orderable group has infinite abelianization.
\end{proof}

\begin{proposition} \label{obstruction_spectrum_p}
Let $p$ be a prime with $p\ge 2$.  There exists a torsion-free circularly-orderable group $G$ such that $\Ob(G) = p\NN$. 

\end{proposition}
\begin{proof}
Let $H\cong \mathbb{Z}[1/(p+1)]^p$ denote the subgroup of the units group of 
$$\mathbb{C}[x_1^{\pm 1/(p+1)^j}, \ldots ,x_p^{\pm 1/(p+1)^j}\colon j\ge 0]$$ generated by $x_1^{1/(p+1)^j},\ldots , x_p^{1/(p+1)^j}$ with $j\ge 0$, and let $K$ denote the free abelian group of rank $p$ with generators $y_1,\ldots ,y_p$.  To help in describing the construction, we take $x_{p+i}=x_i$ and $y_{p+i}=y_i$ for $i=1,\ldots ,p$.
Then we form a semidirect product $H\rtimes K$ by declaring that 
$y_i x_j y_i^{-1} = x_j$ if $j\not\in \{i,i+1\}$ and $y_i x_i y_i^{-1} =x_i^{p+1}$ and $y_i x_{i+1}y_i^{-1}=x_{i+1}^{1/(p+1)}$ for $i=1,\ldots ,p$.  
Notice that the subgroup of $K$ generated by $y:=y_1y_2\cdots y_p$ commutes with each $x_i$ and hence we can form the semidirect product $N:=H\rtimes K/(y)$.  

Now we make an automorphism $f$ of $H\rtimes K$ by declaring 
$f(x_i) = x_{i+1}$ and $f(y_i) = y_{i+1}$. Then since the element $y$ is fixed by $f$, we see that $f$ induces an order-$p$ automorphism of $N$ and so we have the semidirect product $B:=N\rtimes \langle z | z^p=1\rangle$.

There is a surjective homomorphism $\phi: H\to \langle z | z^p=1\rangle$, which sends  
$x_i^{1/(p+1)^j}$ to $z$ for each $i,j$.  Then we take 
$$G = \{ h k z^i \mid h\in H, k\in K/(y), i\in \{0,1,\ldots, p-1\}\text{ and }\phi(h)=z^i\}.$$
By construction, $\phi(y_i h y_i^{-1}) = \phi(h)$ for $h\in H$ and $i=1,\ldots , p$ and $\phi(zhz^{-1})=\phi(h)$ for $h\in H$. Thus we see $G$ is closed under products and taking inverses and so it is a subgroup of $N\rtimes \langle z | z^p=1\rangle$.

We next claim that $G$ is torsion-free.  To see this, suppose that 
$hk z^i\in G$ is a torsion element.  Then as $(hkz^i)^p \in N$ and since $N$ is torsion-free, we see that $hkz^i$ must have order $p$ and $i\in \{1,\ldots ,p-1\}$.  By replacing $hkz^i$ by a suitable power, we may assume that $i=1$ and so $\phi(h)=z$.  Then there exists some $h'\in H$ and $k'\in K$ such that $(hk z)^p = h' k'$; moreover, a straightforward computation shows that $k' = k(zkz^{-1})(z^2 k z^{-2})\cdots (z^{p-1} k z^{-1+p})$.  Since $k$ is the image in $K/(y)$ of an element of the form $y_1^{b_1}\cdots y_p^{b_p}$, we see that
$k(zkz^{-1})(z^2 k z^{-2})\cdots (z^{p-1} k z^{-1+p}) = y_1^{B}\cdots y_p^{B}=y^B$, where $B=\sum b_i$.  In particular, $k'=1$ in $K/(y)$.  We also have that $h' = h \tau(h)\cdots \tau^{p-1}(h)$, where $\tau$ is the automorphism induced by conjugation by $y_1^{b_1}\cdots y_p^{b_p}z$.
This automorphism sends $x_i$ to $x_{i+1}^{(p+1)^{b_{i+1}-b_i}}$ and thus if $h=x_1^{a_1}\cdots x_p^{a_p}$, then 
$$h\tau(h)\cdots \tau^{p-1}(h)= \prod_{j=1}^p x_j^{a_j + a_{j-1}(p+1)^{b_j-b_{j-1}}+a_{j-2}(p+1)^{b_j-b_{j-2}} + \cdots + a_{j-p+1} (p+1)^{b_j - b_{j-p+1}}},$$ where $a_{i}=a_{i+p}, b_{i}=b_{i+p}$ for $i\le 0$.
Then we see that if $h'=1$ then we must have
$$a_j + a_{j-1}(p+1)^{b_j-b_{j-1}}+\cdots +a_{j-p+1} (p+1)^{b_j - b_{j-p+1}} =0$$ for every $j$.
Looking mod $p$, where we take $(p+1)^j=1$ mod $p$ for $j<0$, we see that 
$a_1+\cdots + a_p=0~(\bmod ~p)$, which gives that $\phi(h)=1$, a contradiction, since $\phi(h)=z$.

Note that solvable groups are amenable. Therefore by Propositions \ref{daves reduction} and \ref{prop: Adamsreduction}, to complete the proof we must show that $G$ is a finitely generated solvable circularly-orderable group and show that $G/G'$ is an elementary abelian $p$-group. We first show that $G$ is finitely generated. By construction, $x_i^p\in G$ for all $i$ and the relation $y_i x_i^p y_i^{-1} = x_i^{(p+1)p}$ holds.  Thus if we let $G_0$ denote the finitely generated subgroup of $B:=N\rtimes \langle z|z^p=1\rangle$ generated by $x_1^p,\ldots ,x_p^p, y_1,\ldots ,y_p$, then $G_0$ is normal in $B$ and by construction $B/G_0$ is a finite $p$-group.  In fact, it is a semidirect product 
$\left(\mathbb{Z}/p\mathbb{Z}\right)^p\rtimes \langle z\rangle$, where $z$ acts on $\left(\mathbb{Z}/p\mathbb{Z}\right)^p$ via cyclic permutation. Since $G/G_0$ embeds in $B/G_0$ we see that $G/G_0$ is a finite $p$-group and so $G$ is finitely generated since $G_0$ is finitely generated. Since $G$ is a subgroup of the solvable group $B$, we also have that $G$ is solvable. 

Next to show that $G/G'$ is an elementary abelian $p$-group, first note that $g_{i,j}:=x_i^{1/(p+1)^j} x_{i+1}^{-1/(p+1)^j}\in G$ for $i\in \{1,\ldots ,p\}$ and $j\in \mathbb{Z}$.  Then $$[g_{i,j}, y_{i+1}]=x_i^{1/(p+1)^j} x_{i+1}^{-1/(p+1)^j} y_{i+1} x_{i+1}^{1/(p+1)^j} x_i^{-1/(p+1)^j} y_{i+1}^{-1}\in G'.$$ 
A straightforward computation then yields that
$ x_{i+1}^{p/(p+1)^j}\in G'$. Next observe that $x_i z, x_{i+1}z \in G$ and their commutator is equal to
$$x_i z x_{i+1} z z^{-1} x_i^{-1} z^{-1} x_{i+1}^{-1} = x_i z x_{i+1} x_i^{-1} z^{-1} x_{i+1}^{-1} = x_i x_{i+1}^{-2} x_{i+2}.$$  More generally we have
$x_i^{(p+1)^j} x_{i+1}^{-2(p+1)^j} x_{i+2}^{(p+1)^j}\in G'$ for all $j\in \mathbb{Z}$.  We also have $[y_i, z]=y_i z y_i z^{-1} = y_i y_{i+1}^{-1}\in G'$, where we take $y_{p+1}=y_1$.  Since every element of $K$ can be reduced using the relations $y_i y_{i+1}^{-1}$, with $i=1,2,\ldots ,p$, and $y$, to obtain an element of the form $y_p^{\ell}$ with $0\le \ell<p$, we see $\left(K/(y) \right)/(K/(y)\cap G')$ is a cyclic group of order $p$ that is generated by the image of $y_p$.

Notice that $G$ is normal in $B$ and $G'$ is characteristic in $G$ and hence $G'$ is normal in $B$. If $hk z^i\in B$, with $h\in H$, $k\in K/(y)$, then from the above remark, modulo $G'$, $hkz^i$ is equal to $h y_p^{\ell} z^i$ for some $\ell\in \{0,1,\ldots ,p-1\}$.  Moreover, using the fact that $x_i^{(p+1)^j} x_{i+1}^{-2(p+1)^j} x_{i+2}^{(p+1)^j}\in G'$ for all $j\in \mathbb{Z}$ and $i=1,\ldots ,p-1$, we see that modulo $G'$ that we can take $h$ to be of the form $x_1^a x_2^b$ with $a,b\in \mathbb{Z}[1/(p+1)]$.  Since $x_1^{p/(p+1)^j}, x_2^{p/(p+1)^j}\in G'$ for all $j\in \mathbb{Z}$, we see that we can in fact take $a,b \in \{0,1,\ldots, p-1\}$ and so $B/G'$ is a homomorphic image of a semidirect product $$\left(\mathbb{Z}/p\mathbb{Z}\right)^3 \rtimes \langle z|z^p=1\rangle$$ (with the first two copies of $\mathbb{Z}/p\mathbb{Z}$ being generated by the images of $x_1$ and $x_2$ and the last copy being generated by the image of $y_p$), which is a nonabelian group of order $p^3$ in which the automorphism of $\left(\ZZ/p\ZZ\right)^3$ induced by conjugation by $z$ is given by 
$(1,0,0)\mapsto (0,1,0)$ and $(0,1,0)\mapsto (-1,2,0)$, and this group is a nonabelian group of order $p^4$ in which every element has order $p$.  Since $G/G'$ embeds in this group and is abelian and since $G$ is solvable, it is a non-trivial proper subgroup of this group and hence is an elementary abelian $p$-group of of order in $\{p,p^2,p^3\}$.

It now only remains to show that $G$ is circularly-orderable. Observe first that $H$ is left-orderable. Moreover, $K/(y)$ is left-orderable, as it is $\mathbb{Z}^{p-1}$, since the element $y$ corresponds to a unimodular row in $\mathbb{Z}^p$.  Thus $N$ is left-orderable, as it is a semidirect product of two left-orderable groups. This then gives that $B=N\rtimes \langle z|z^p=1\rangle$ is circularly-orderable as it is a semidirect product of a left-orderable group and a circularly-orderable group. Hence $G$ is circularly-orderable as it is a subgroup of $B$. Finally, in light of Proposition \ref{prop: Adamsreduction} we get that $G\times \ZZ/p\ZZ$ is not circularly-orderable.

Suppose now $\gcd(m,p) = 1$.   Notice that $G\times \ZZ/m\ZZ$ is a subgroup of $B\times \ZZ/m\ZZ$ and since $B\times \ZZ/m\ZZ$ fits into a short exact sequence
$$1\to H\rtimes K/(y)\to B\times \ZZ/m\ZZ\to \ZZ/p\ZZ\times \ZZ/m\ZZ \to 0,$$ we see that $B\times \ZZ/m\ZZ$ is circularly-orderable, as $H\rtimes K/(y)$ is left-orderable and $\ZZ/p\ZZ \times \ZZ/m\ZZ \cong \ZZ/pm\ZZ$ is circularly-orderable.  It follows that $G\times \ZZ/m\ZZ$ is circularly-orderable, and therefore $\Ob(G) = p\NN$.
\end{proof}

Next we show that the obstruction spectrum of the free product of circularly-orderable groups is equal to the union of the obstruction spectra of each of the component groups.  Note that it makes sense to talk about the obstruction spectrum of a free product of circularly-orderable groups since such a free product is circularly-orderable \cite{BS}.

\begin{lemma}\label{obstruction_free_products}
Let $\{G_\alpha \mid \alpha \in \mathcal A\}$ be a set of circularly-orderable groups.  Then $\Ob(*G_\alpha) = \bigcup_{\alpha \in \mathcal A} \Ob(G_\alpha)$.
\end{lemma}
\begin{proof}
Let $k \in \NN_{>1}$.  Consider the amalgamated free product $*_{\ZZ/k\ZZ} (G_\alpha \times \ZZ/k\ZZ)$ defined by the injective maps $\iota_\alpha:\ZZ/k\ZZ \to G_\alpha \times \ZZ/k\ZZ$ where $\iota_\alpha(t) = (\text{id},t_\alpha)$ for a choice of generator $t \in \ZZ/k\ZZ$.  The homomorphisms $\phi_\alpha:G_\alpha \times \ZZ/k\ZZ \to (*G_\alpha)\times \ZZ/k\ZZ$ given by $\phi_\alpha(g,t_\alpha^m) = (g,s^m)$ (where $s \in \ZZ/k\ZZ$ is a chosen generator) induce an isomorphism
\[
*_{\ZZ/k\ZZ} (G_\alpha \times \ZZ/k\ZZ)\cong(*G_\alpha) \times \ZZ/k\ZZ.
\]
It now suffices to show $*_{\ZZ/k\ZZ} (G_\alpha \times \ZZ/k\ZZ)$ is circularly-orderable if and only if $G_\alpha \times \ZZ/k\ZZ$ is circularly-orderable for all $\alpha$.

If $*_{\ZZ/k\ZZ} (G_\alpha \times \ZZ/k\ZZ)$ is circularly-orderable, it is clear that $G_\alpha \times \ZZ/k\ZZ$ is circularly-orderable for all $\alpha$.  Conversely, suppose $G_\alpha \times \ZZ/k\ZZ$ is circularly-orderable for all $\alpha$.  By applying an automorphism of $G_\alpha \times \ZZ/k\ZZ$ of the form $(g,t_\alpha^m) = (g,\varphi(t_\alpha^m))$, where $\varphi \in \operatorname{Aut}(\ZZ/k\ZZ)$, we may assume there is a circular ordering $c$ on $\ZZ/k\ZZ$ and circular orderings $c_\alpha$ on $G_\alpha \times \ZZ/k\ZZ$ such that $\iota_\alpha^*(c_\alpha) = c$.  Since $\ZZ/k\ZZ$ is cyclic, $*_{\ZZ/k\ZZ} (G_\alpha \times \ZZ/k\ZZ)$ is circularly-orderable by \cite[Proposition 1.2]{CG}, completing the proof.
\end{proof}

\begin{proposition}
For every $S \subset \Pi$, there exists a torsion-free circularly-orderable group $G$ such that $\Ob(G) = \bigcup_{p \in S}p\NN$.
\end{proposition}
\begin{proof}
The result follows by combining Proposition \ref{obstruction_spectrum_p} with Lemma \ref{obstruction_free_products}.
\end{proof}

We end this section with an example of a torsion-free group $G$ such that $\Ob(G) \neq \bigcup_{p \in S} p \NN$ for any set of primes $S$.  Thus the subsets of $\NN$ that arise as $\Ob(G)$ for some torsion-free group $G$ are more varied in structure than the sets which arise as $\Ob_T(G)$ for some group $G$.

\begin{example}
Consider the Promislow group $G:= \langle a, b \mid ab^{2}a^{-1} = b^{-2}, ba^2b^{-1}=a^{-2} \rangle$, also known as the fundamental group of the Hantzsche-Wendt manifold.  The group $G$ is torsion-free and not left-orderable, so $\Ob(G)$ is nonempty (\cite{Promislow88}, see also \cite[Example 1.59]{CR16}).  

The group $G$ is circularly-orderable.  There is a map $\phi : G \rightarrow \ZZ/2\ZZ$ given by $\phi(a) =1, \phi(b)=0$ whose kernel is the subgroup $\langle b, a^2, (ab)^2\rangle$. One can verify that $\langle b, a^2, (ab)^2\rangle \cong \ZZ^2 \ltimes \ZZ$, where $\ZZ^2$ is generated by $a^2, (ab)^2$ and $\ZZ$ is generated by $b$.  In particular, the kernel of $\phi$ is left-orderable so that $G$ is circularly-orderable by a lexicographic argument.  

Note that as a consequence of the circular ordering constructed in the previous paragraph, $n \notin \Ob(G)$ whenever $n \in \NN_{>1}$ is odd.  Indeed, when $n$ is odd $(\phi \times \id): G \times \ZZ/n\ZZ \rightarrow \ZZ/2\ZZ \times \ZZ/n\ZZ \cong \ZZ/2n\ZZ$ provides a map of $G \times \ZZ/n\ZZ$ onto a circularly-orderable group with kernel isomorphic to the left-orderable group $\ker(\phi)$.  Thus $G \times \ZZ/n\ZZ$ is circularly-orderable.

However, $2$ is also not in the obstruction spectrum of $G$.  To see this, note that the abelianization of $G$ is $(\ZZ/4\ZZ)^2$ with abelianization map $a \mapsto (1,0)$ and $b \mapsto(0,1)$.  Let $\psi : G \rightarrow \ZZ/4\ZZ$ denote the composition of the abelianization map with projection onto the first factor, so $\psi(a) = 1, \psi(b)=0$, and let $\iota:\ZZ/2\ZZ \rightarrow \ZZ/4\ZZ$ denote the obvious inclusion.  Define $\beta : G \times \ZZ/2\ZZ \rightarrow \ZZ/4\ZZ$ by $\beta(g, a) = \psi(g) + \iota(a)$.  We claim that $\ker(\beta) \cong \ker(\phi)$, which will show that $\ker(\beta)$ is left-orderable so that $G \times \ZZ/2\ZZ$ is circularly-orderable and $2 \notin \Ob(G)$.  The claim follows from checking that the map $\alpha : \ker(\beta) \rightarrow \ker(\phi)$ given by $\alpha(g, a) = g$ is an isomorphism.  Further, whenever $n$ is odd the map $(\beta \times \id): (G \times \ZZ/2\ZZ) \times \ZZ/n\ZZ \rightarrow \ZZ/4\ZZ \times \ZZ/n\ZZ$ yields a map onto a cyclic group whose kernel is isomorphic to $\ker(\beta)$.  Thus $2n \notin \Ob(G)$ whenever $n$ is odd, and we conclude $\Ob(G) \subset 4\NN$.

Last, note that the subgroup $\langle a^2, b^2, (ab)^2 \rangle$ is isomorphic to $\ZZ^3$ and is normal, with quotient $\ZZ/2\ZZ \times \ZZ/2\ZZ$.  In particular this tells us that $G$ is amenable.  Since $G/G' \cong (\ZZ/4\ZZ)^2$, it follows from Proposition \ref{prop: Adamsreduction} that $4 \in \Ob(G)$ and so $\Ob(G) = 4\NN$.
\end{example}

\subsection{Abelianization and the obstruction spectrum}

Suppose $G$ is a circularly-orderable but not left-orderable group. Since it is not left-orderable, there exists a finitely generated subgroup $H < G$ with no left-orderable quotients \cite{BH72}, and therefore the abelianization $H/H'$ is finite. The main result of this section is Corollary \ref{corollary-of-dave}, which is a connection between the exponent of $H/H'$ and the obstruction spectrum of $G$. Corollary \ref{corollary-of-dave} is a direct consequence of Theorem \ref{theorem-of-dave} whose statement and proof are due to Dave Morris.

\begin{theorem}[Morris]
\label{theorem-of-dave}
Let $H$ be a non-trivial, finitely generated, circularly-orderable group with no non-trivial left-orderable quotients. Let $e = \exp(H/H')$ be the exponent of the abelianization of $H$. Then for any integer $m \geq 2$, there exists a $k \geq 0$ such that $H \times \ZZ/(em^k)\ZZ$ is not circularly-orderable.
\end{theorem}
\begin{proof}
Let $Q = \frac{1}{e}\ZZ\left[\frac 1m\right]$. The first step is to show $H \times Q/\ZZ$ is not circularly-orderable. Suppose towards a contradiction, that $H^* = H \times Q/\ZZ$ is circularly-orderable. Fix some circular ordering on $H^*$ and consider the central extension
\[
1 \lra \ZZ \overset{\iota}{\lra} \wt{H^*} \overset{\rho}{\lra} H^* \lra 1
\]
arising from the given circular ordering as in Section \ref{background}. In particular, $\wt{H^*}$ is left-orderable. Let $R = \rho^{-1}(Q/\ZZ)$, and let $K$ be the intersection of all relatively convex subgroups containing $R$. Since $R$ is normal, we have that $K$ is normal. Then $\wt{H^*}/K \cong H^*/\rho(K) \cong H/\pi_1\rho(K)$ where $\pi_1: H \times Q/\ZZ \to H$ is the projection map. Therefore $\wt{H^*}/K$ is a left-orderable group isomorphic to a quotient of $H$, which implies $\wt{H^*} = K$.

Note that $R \cong Q$ by Proposition \ref{divisible lift}, so $R$ is isomorphic to a dense subgroup of $\QQ$. Therefore by Lemma \ref{normal-infinitesimals}, there exists a normal subgroup $M$ such that $M \cap R = \{id\}$ and $\wt{H^*}/M$ is isomorphic to a subgroup of $\QQ$. Let $\tilde \pi:\wt{H^*} \to \wt{H^*}/M$ be the quotient map. Since $M \cap R = \{id\}$, $R$ maps injectively into $\wt{H^*}/M$ under $\tilde \pi$.

Let $\pi:H^* \to H^*/\rho(M)$ be the quotient map. Then $\rho(R) = Q/\ZZ$ maps injectively into $H^*/\rho(M)$ under $\pi$. Note that since $\wt{H^*}/M$ is isomorphic to a subgroup of $\QQ$, $H^*/\rho(M)$ is isomorphic to a subgroup of $\QQ/\ZZ$. Furthermore, since $\QQ/\ZZ$ is abelian and the exponent of $H/H'$ is $e$, $\pi(H)$ must be contained in the unique subgroup of $H^*/\rho(M)$ of order $e$. By the construction of $Q$, $Q/\ZZ$ contains a subgroup of order $e$, and therefore $\pi(H) \subset \pi(Q/\ZZ)$. Therefore $\pi|_{Q/\ZZ}:Q/\ZZ \to H^*/\rho(M)$ is an isomorphism. It follows that $\tilde\pi|_R:R \to \wt{H^*}/M$ is an isomorphism. By Lemma \ref{cofinal-central-subgroup}, $R$ is central in $\wt{H^*}$ and therefore $\wt{H^*} \cong R \times M$. Since $\wt{H^*}$ is smallest relatively convex subgroup containing $R$, we conclude that $M = \{id\}$. However, this implies $H$ is trivial, a contradiction.

We now have that $H \times Q/\ZZ$ is not circularly-orderable. Therefore there exists a finitely generated subgroup of $H \times Q/\ZZ$ that is not circularly-orderable \cite[Lemma 2.14]{Clay}. Every finitely generated subgroup of $H \times Q/\ZZ$ is contained in a subgroup of the form $H \times T$, where $T$ is a non-trivial finitely generated subgroup of $Q/\ZZ$. However, all finitely generated subgroups of $Q/\ZZ$ are finite cyclic groups. By the construction of $Q$, for every finite cyclic subgroup $S \subset Q/\ZZ$ there exists a $k$ such that $S$ is contained in the unique cyclic subgroup of $Q/\ZZ$ of order $em^k$. Therefore there exists a $k \geq 0$ such that $H \times \ZZ/(em^k)\ZZ$ is not circularly-orderable.
\end{proof}

Now given an arbitrary non-left orderable group $G$, there exists a finitely generated non-trivial subgroup $H$ that has no non-trivial left-orderable quotients \cite{BH72}. Therefore we can apply Theorem \ref{theorem-of-dave} to arbitrary groups to obtain the next corollary.

\begin{corollary}
\label{corollary-of-dave}
Let $G$ be a circularly-orderable but not left-orderable group, and let $H$ be a non-trivial finitely generated subgroup with no non-trivial left-orderable quotients. Let $e = \exp(H/H')$ be the exponent of the abelianization of $H$. Then for any integer $m \geq 2$, there exists a $k \geq 0$ such that $em^k \in \Ob(G)$.
\end{corollary}

\begin{example}
Consider the mapping class group $G = \operatorname{Mod}(\Sigma_{g,1})$ of an orientable genus $g\geq 3$ surface with one puncture (see \cite{FM} for an introduction to mapping class groups), which is circularly-orderable but not left-orderable since it contains torsion. Furthermore, since $G$ is generated by torsion elements \cite{Korkmaz}, it has no non-trivial left-orderable quotients. By a result of Harer \cite{Harer}, $G/G' = \{id\}$. Therefore by Theorem \ref{theorem-of-dave}, for every integer $m \geq 2$, there exists a $k_m \geq 1$ such that $m^{k_m} \in \Ob(G)$.

However, the maximum order of an element in $\operatorname{Mod}(\Sigma_{g,1})$ is bounded above. Indeed, if we combine the $4g+2$ theorem \cite{Wiman} with the Nielsen realization theorem \cite{Kerckhoff}, we obtain that the maximum order of an element in $\operatorname{Mod}(\Sigma_g)$ (where $\Sigma_g$ is a closed surface of genus $g$) is $4g+2$ (see \cite[Section 7.2]{FM} for a self-contained proof of this fact). Since $\pi_1(\Sigma_g)$ is torsion free, we can use the Birman exact sequence \cite{Birman}
\[
1 \lra \pi_1(\Sigma_g) \lra \operatorname{Mod}(\Sigma_{g,1}) \lra \operatorname{Mod}(\Sigma_g) \lra 1
\]
to conclude that the maximum order of an element in $\operatorname{Mod}(\Sigma_{g,1})$ is also $4g+2$.

Therefore there are elements in the obstruction spectrum of $G$ that do not detect torsion. More formally, the torsion part of the obstruction spectrum $\Ob_T(G)$ is a proper subset of the obstruction spectrum $\Ob(G)$.
\end{example}

The results in Section \ref{obstruction section} give some information as to the possible structure of obstruction spectra. However, it is not a complete picture, and we leave the reader with the natural question:

\begin{question}
Can every set of the form $\bigcup_{m \in S}m\NN$ where $S \subset \NN_{>1}$ arise as $\Ob(G)$ for some group $G$?
\end{question}

\bibliographystyle{plain}

\bibliography{loandcirc}

\begin{thebibliography}{10}

\bibitem{BS}
Hyungryul Baik and Eric Samperton.
\newblock Space of invariant circular orders of groups.
\newblock {\em Groups Geom. Dyn.}, 12(2):721--763, 2018.

\bibitem{Birman}
Joan~S. Birman.
\newblock Mapping class groups and their relationship to braid groups.
\newblock {\em Comm. Pure Appl. Math.}, 22:213--238, 1969.

\bibitem{BGW13}
Steven Boyer, Cameron~McA. Gordon, and Liam Watson.
\newblock On {L}-spaces and left-orderable fundamental groups.
\newblock {\em Math. Ann.}, 356(4):1213--1245, 2013.

\bibitem{BH18}
Steven Boyer and Ying Hu.
\newblock Taut foliations in branched cyclic covers and left-orderable groups.
\newblock {\em Trans. Amer. Math. Soc.}, 372(11):7921--7957, 2019.

\bibitem{BRW05}
Steven Boyer, Dale Rolfsen, and Bert Wiest.
\newblock Orderable 3-manifold groups.
\newblock {\em Ann. Inst. Fourier (Grenoble)}, 55(1):243--288, 2005.

\bibitem{BH72}
R.~G. Burns and V.~W.~D. Hale.
\newblock A note on group rings of certain torsion-free groups.
\newblock {\em Canad. Math. Bull.}, 15:441--445, 1972.

\bibitem{Calegari}
Danny Calegari.
\newblock Circular groups, planar groups, and the {E}uler class.
\newblock In {\em Proceedings of the {C}asson {F}est}, volume~7 of {\em Geom.
  Topol. Monogr.}, pages 431--491. Geom. Topol. Publ., Coventry, 2004.

\bibitem{Calegari06}
Danny Calegari.
\newblock Dynamical forcing of circular groups.
\newblock {\em Trans. Amer. Math. Soc.}, 358(8):3473--3491, 2006.

\bibitem{CD03}
Danny Calegari and Nathan~M. Dunfield.
\newblock Laminations and groups of homeomorphisms of the circle.
\newblock {\em Invent. Math.}, 152(1):149--204, 2003.

\bibitem{CGH16}
Katherine Christianson, Justin Goluboff, Linus Hamann, and Srikar Varadaraj.
\newblock Non-left-orderable surgeries on twisted torus knots.
\newblock {\em Proc. Amer. Math. Soc.}, 144(6):2683--2696, 2016.

\bibitem{Clay}
Adam Clay.
\newblock Generalizations of the {B}urns--{H}ale theorem.
\newblock {\em Comm. Algebra}, 48(11):4846--4858, 2020.

\bibitem{CG}
Adam Clay and Tyrone Ghaswala.
\newblock Free products of circularly ordered groups with amalgamated subgroup.
\newblock {\em J. Lond. Math. Soc. (2)}, 100(3):775--803, 2019.

\bibitem{CR16}
Adam Clay and Dale Rolfsen.
\newblock {\em Ordered groups and topology}, volume 176 of {\em Graduate
  Studies in Mathematics}.
\newblock American Mathematical Society, Providence, RI, 2016.

\bibitem{CW13}
Adam Clay and Liam Watson.
\newblock Left-orderable fundamental groups and {D}ehn surgery.
\newblock {\em Int. Math. Res. Not. IMRN}, (12):2862--2890, 2013.

\bibitem{Conrad59}
Paul Conrad.
\newblock Right-ordered groups.
\newblock {\em Michigan Math. J.}, 6:267--275, 1959.

\bibitem{CD18}
Marc Culler and Nathan~M. Dunfield.
\newblock Orderability and {D}ehn filling.
\newblock {\em Geom. Topol.}, 22(3):1405--1457, 2018.

\bibitem{EHN81}
David Eisenbud, Ulrich Hirsch, and Walter Neumann.
\newblock Transverse foliations of {S}eifert bundles and self-homeomorphism of
  the circle.
\newblock {\em Comment. Math. Helv.}, 56(4):638--660, 1981.

\bibitem{FM}
Benson Farb and Dan Margalit.
\newblock {\em A primer on mapping class groups}, volume~49 of {\em Princeton
  Mathematical Series}.
\newblock Princeton University Press, Princeton, NJ, 2012.

\bibitem{Ghys01}
\'{E}tienne Ghys.
\newblock Groups acting on the circle.
\newblock {\em Enseign. Math. (2)}, 47(3-4):329--407, 2001.

\bibitem{Harer}
John Harer.
\newblock The second homology group of the mapping class group of an orientable
  surface.
\newblock {\em Invent. Math.}, 72(2):221--239, 1983.

\bibitem{NJ}
Mark Jankins and Walter~D. Neumann.
\newblock {\em Lectures on {S}eifert manifolds}, volume~2 of {\em Brandeis
  Lecture Notes}.
\newblock Brandeis University, Waltham, MA, 1983.

\bibitem{JN85}
Mark Jankins and Walter~D. Neumann.
\newblock Rotation numbers of products of circle homeomorphisms.
\newblock {\em Math. Ann.}, 271(3):381--400, 1985.

\bibitem{Ju15}
Andr\'{a}s Juh\'{a}sz.
\newblock A survey of {H}eegaard {F}loer homology.
\newblock In {\em New ideas in low dimensional topology}, volume~56 of {\em
  Ser. Knots Everything}, pages 237--296. World Sci. Publ., Hackensack, NJ,
  2015.

\bibitem{Kerckhoff}
Steven~P. Kerckhoff.
\newblock The {N}ielsen realization problem.
\newblock {\em Ann. of Math. (2)}, 117(2):235--265, 1983.

\bibitem{KL}
Dongseok Kim and Jaeun Lee.
\newblock Some invariants of pretzel links.
\newblock {\em Bull. Austral. Math. Soc.}, 75(2):253--271, 2007.

\bibitem{Kokorin}
A.~I. Kokorin.
\newblock Intersection and union of relatively convex subgroups of orderable
  groups.
\newblock {\em Algebra i Logika}, 7(3):48--50, 1968.

\bibitem{KM96}
Valeri\u{i}~M. Kopytov and Nikola\u{i}~Ya. Medvedev.
\newblock {\em Right-ordered groups}.
\newblock Siberian School of Algebra and Logic. Consultants Bureau, New York,
  1996.

\bibitem{Korkmaz}
Mustafa Korkmaz.
\newblock Generating the surface mapping class group by two elements.
\newblock {\em Trans. Amer. Math. Soc.}, 357(8):3299--3310, 2005.

\bibitem{Los54}
J.~\L~os.
\newblock On the existence of linear order in a group.
\newblock {\em Bull. Acad. Polon. Sci. Cl. III.}, 2:21--23, 1954.

\bibitem{LM12}
Peter Linnell and Dave Witte~Morris.
\newblock Amenable groups with a locally invariant order are locally indicable.
\newblock {\em Groups Geom. Dyn.}, 8(2):467--478, 2014.

\bibitem{Naimi94}
Ramin Naimi.
\newblock Foliations transverse to fibers of {S}eifert manifolds.
\newblock {\em Comment. Math. Helv.}, 69(1):155--162, 1994.

\bibitem{Nie18}
Zipei Nie.
\newblock Left-orderablity for surgeries on {$(-2,3,2s+1)$}-pretzel knots.
\newblock {\em Topology Appl.}, 261:1--6, 2019.

\bibitem{Ohnishi52}
Masao Ohnishi.
\newblock Linear-order on a group.
\newblock {\em Osaka Math. J.}, 4:17--18, 1952.

\bibitem{Promislow88}
S.~David Promislow.
\newblock A simple example of a torsion-free, nonunique product group.
\newblock {\em Bull. London Math. Soc.}, 20(4):302--304, 1988.

\bibitem{Rolfsen04}
Dale Rolfsen.
\newblock Mappings of nonzero degree between 3-manifolds: a new obstruction.
\newblock In {\em Advances in topological quantum field theory}, volume 179 of
  {\em NATO Sci. Ser. II Math. Phys. Chem.}, pages 267--273. Kluwer Acad.
  Publ., Dordrecht, 2004.

\bibitem{Sw59}
S.~\'{S}wierczkowski.
\newblock On cyclically ordered groups.
\newblock {\em Fund. Math.}, 47:161--166, 1959.

\bibitem{Thurston}
William~P. Thurston.
\newblock Three-dimensional manifolds, {K}leinian groups and hyperbolic
  geometry.
\newblock {\em Bull. Amer. Math. Soc. (N.S.)}, 6(3):357--381, 1982.

\bibitem{Wiman}
A.~Wiman.
\newblock Unber die hyperelliptischen curven and diejenigan vom geschlechte $p
  = 3$, welche eindeutigen transformationen in sich zulassen.
\newblock {\em Bihang Kongl. Svenska Vetenskaps-Akademiens Handlingar},
  1895-1896.

\bibitem{zheleva76}
S.~D. Zeleva.
\newblock Cyclically ordered groups.
\newblock {\em Sibirsk. Mat. \v Z.}, 17(5):1046--1051, 1197, 1976.

\end{thebibliography}

\end{document}